\documentclass{article}
\usepackage{amsmath}
\usepackage{amsthm}
\usepackage{amssymb}

\theoremstyle{definition}
\newtheorem{defi}{Definition}

\theoremstyle{plain}
\newtheorem{thm}[defi]{Theorem}
\newtheorem{cor}[defi]{Corollary}

\newtheorem{prop}[defi]{Proposition}
\newtheorem{ques}[defi]{Question}

\newtheorem{ex}[defi]{Example}

\theoremstyle{remark}
\newtheorem{sclaim}{Claim}[defi]

\newcounter{enuroman}
\renewcommand{\theenuroman}{\roman{enuroman}}

\newcounter{enuRoman}
\renewcommand{\theenuRoman}{(\Roman{enuRoman})}

\newcounter{enuAlph}
\renewcommand{\theenuAlph}{\Alph{enuAlph}}

\newcounter{enualph}
\renewcommand{\theenualph}{\alph{enualph}}

\newcounter{enuarabic}
\renewcommand{\theenuarabic}{\arabic{enuarabic}}

\newcommand{\re}{{\upharpoonright}}

\newcommand{\I}{{\cal I}}
\newcommand{\J}{{\cal J}}

\newcommand{\X}{{\cal X}}

\newcommand{\FM}{{\mathbb{FM}}}

\newcommand{\LL}{{\mathbb L}}
\newcommand{\MM}{{\mathbb M}}

\newcommand{\RR}{{\mathbb R}}
\renewcommand{\SS}{{\mathbb S}}

\newcommand{\TT}{{\mathbb T}}
\newcommand{\VV}{{\mathbb V}}

\newcommand{\bb}{{\mathfrak b}}

\newcommand{\cc}{{\mathfrak c}}
\newcommand{\dd}{{\mathfrak d}}

\newcommand{\add}{{\mathsf{add}}}
\newcommand{\cov}{{\mathsf{cov}}}
\newcommand{\non}{{\mathsf{non}}}
\newcommand{\cof}{{\mathsf{cof}}}

\newcommand{\stem}{{\mathrm{stem}}}

\newcommand{\dom}{{\mathrm{dom}}}

\renewcommand{\succ}{{\mathrm{succ}}}

\newcommand{\sub}{\subseteq}
\newcommand{\sem}{\setminus}
\newcommand{\twoom}{2^\omega}
\newcommand{\twolom}{2^{<\omega}}
\newcommand{\omlom}{\omega^{<\omega}}
\newcommand{\omom}{\omega^\omega}

\newcommand{\ha}{\,{}\hat{}\,}

\newcommand{\Loleriar}{\mbox{$\Longleftrightarrow$}}

\title{Cofinalities of Marczewski-like ideals}

\author{J\"org Brendle\thanks{Partially supported by Grants-in-Aid for Scientific Research
   (C) 24540126 and (C) 15K04977, Japan Society for the Promotion of Science, 
   by JSPS and FWF under the Japan-Austria Research Cooperative Program 
   {\em New developments regarding forcing in set theory}, and
   by the Isaac Newton Institute for Mathematical Sciences in the programme {\em Mathematical, 
   Foundational and Computational Aspects of the Higher Infinite} (HIF) funded by EPSRC grant EP/K032208/1.}  \\
   Graduate School of System Informatics \\
   Kobe University \\
   Rokko-dai 1-1, Nada-ku \\
   Kobe 657-8501, Japan \\
   email: {\sf brendle@kobe-u.ac.jp} \\  \\
   and \\ \\
   Yurii Khomskii\thanks{Partially supported by the Isaac Newton Institute for Mathematical Sciences in the programme {\em Mathematical, 
   Foundational and Computational Aspects of the Higher Infinite} (HIF) funded by EPSRC grant EP/K032208/1.\newline
   } \\  
   Hamburg University \\
   Bundesstra{\ss}e 55 (Geomatikum) \\
   20146 Hamburg, Germany \\
   email: {\sf yurii@deds.nl} \\ \\
   and \\ \\
   Wolfgang Wohofsky\footnotemark[2] \\
   Hamburg University \\
   Bundesstra{\ss}e 55 (Geomatikum) \\
   20146 Hamburg, Germany \\
   email: {\sf wolfgang.wohofsky@gmx.at}
}

\begin{document}
\maketitle

\begin{abstract}
\noindent We show that the cofinalities of both the Miller ideal $m^0$ (the $\sigma$-ideal naturally related to Miller forcing $\MM$) 
and the Laver ideal $\ell^0$ (related to Laver forcing $\LL$) are larger than the size of the continuum $\cc$ in ZFC.
\end{abstract}



\section{Introduction}

The purpose of this note is to prove (in ZFC) that the ideals naturally related to Laver forcing $\LL$ and to Miller forcing
$\MM$, the {\em Laver ideal} $\ell^0$ and the {\em Miller ideal} $m^0$, have cofinality strictly larger than $\cc$, the size of
the continuum  (Corollary~\ref{Laver-Miller-cof} below). We will phrase our result in a more general framework and show that 
$\cof (t^0) > \cc$ holds for all tree ideals $t^0$ derived from tree forcings $\TT$ satisfying a certain property (Theorem~\ref{sdap} in Section 3).
This was known previously for the Marczewski ideal $s^0$~\cite{JMS92} and the nowhere Ramsey ideal $r^0$~\cite{Ma}, but it is unclear whether 
the method of proof for these two ideals works for $\ell^0$ and $m^0$ (see the discussion in Section 2), and our approach is
more general.

For a subtree $T \sub \omlom$ (or $T \sub \twolom$),
 $[T] = \{ x \in \omom : x \re n \in T$ for all $n \}$ denotes the set of branches through $T$.

\begin{defi}[Combinatorial tree forcing]
A collection $\TT$ of subtrees of $\omlom$ is a {\em combinatorial tree forcing} if 
\begin{enumerate}
\item $\omlom \in \TT$,
\item (closure under subtrees) if $T \in \TT$ and $s \in T$, then the tree $T_s = \{ t \in T : s \sub t$ or $t \sub s \}$ also
   belongs to $\TT$,
\item (large disjoint antichains) there is a continuous function $f : \omom \to \twoom$ such that
   for all $x \in \twoom$, $f^{-1} ( \{ x \} )$ is the set of branches of a tree in $\TT$,
\item (homogeneity) if $T \in \TT$, then there is an order-preserving injection $i : \omlom \to T$ such that the map $g: \omom \to [T]$ 
   given by $g(x) = \bigcup \{ i (x \re n) : n \in \omega \}$ is a homeomorphism and for any subtree $S \sub \omlom$,
   $S \in \TT$ iff the downward closure of $i (S)$ belongs to $\TT$.
\end{enumerate}
$\TT$ is ordered by inclusion, that is, for $S,T \in \TT$, $S \leq T$ if $S \sub T$.
\end{defi}

Homogeneity says that the forcing looks the same below each condition. 
In view of homogeneity, ``large disjoint antichains" implies that 
\begin{enumerate} 
\setcounter{enumi}{4}
\item each $T \in \TT$ splits into continuum many trees with pairwise disjoint sets of branches, that is,
there are $T_\alpha \in \TT$, $\alpha < \cc$, with $T_\alpha \sub T$ such that $[T_\alpha] \cap [T_\beta] = \emptyset$
for $\alpha \neq \beta$. 
\end{enumerate}
In particular, there are $\cc$-sized antichains so that $\TT$ is not ccc and forcing notions like Cohen and
random forcing do not fit into this framework.

For forcing notions whose conditions are subtrees of $\twolom$ like Sacks forcing $\SS$, an analogous 
definition applies, with $\omlom$ and $\omom$ replaced by $\twolom$ and $\twoom$, respectively.

\begin{defi}[Tree ideal]
The {\em tree ideal} $t^0$ associated with the combinatorial tree forcing $\TT$ consists of all $X \sub \omom$ such that 
for all $T \in \TT$ there is $S \leq T$ with $X \cap [S] = \emptyset$.
\end{defi}

When $\TT = \SS$ is Sacks forcing, $t^0 = s^0$ is the well-known {\em Marczewski ideal}, and for Mathias
forcing $\TT = \RR$, $t^0 = r^0$ is the ideal of {\em nowhere Ramsey} sets. {\em Laver forcing} $\LL$~\cite{La76} consists
of trees $T \sub \omlom$ such that for all $t \in T$ containing the stem of $T$, $\stem (T)$, the set of successors $\succ_T (t) = \{
n \in \omega : t \ha n \in T \}$ is infinite. {\em Miller forcing} $\MM$~\cite{Mi84} consists of trees $T \sub \omlom$
such that for all $t \in T$ there is $s \supseteq t$ in $T$ such that $\succ_T (s)$ is infinite. 
Note that $\LL$ and $\MM$ are combinatorial tree forcings in the above sense for 
$f : \omom \to \twoom$ given by $f(x) (n) = x  (n) \mod 2$ for all $x \in \omom$ and $n \in \omega$
witnesses ``large disjoint antichains". The Laver and Miller
ideals $\ell^0$ and $m^0$ are the corresponding tree ideals. For basic facts about such tree ideals,
like non-inclusion between different ideals, see e.g.~\cite{Br95}.

\begin{defi}[Cofinality of an ideal]
Given an ideal $\I$, its {\em cofinality} $\cof (\I)$ is the smallest cardinality of a family $\J \sub \I$ such that
every member of $\I$ is contained in a member of $\J$.
\end{defi}

A family like $\J$ in this definition is said to be a {\em basis} of $\I$ (or: {\em cofinal} in $\I$).

While the topic of our work are cofinalities of tree ideals, we note that other cardinal invariants of tree ideals $t^0$,
such as the {\em additivity} $\add (t^0)$ (the least size of a subfamily $\J \sub t^0$ whose union is not in $t^0$) and the
{\em covering number} $\cov (t^0)$ (the least size of a subfamily $\J \sub t^0$ whose union is $\omom$) have been
studied as well.
If there is a fusion argument for $\TT$, $t^0$ is a $\sigma$-ideal, and one has $\omega_1 \leq \add (t^0) 
\leq \cov (t^0) \leq \cc$, while the exact value of these two cardinals depends on the model of set theory.
Furthermore, by ``large disjoint antichains", the {\em uniformity} $\non (t^0)$ of a tree ideal $t^0$
(the smallest cardinality of a subset of $\omom$ not belonging to $t^0$) is always equal to $\cc$.
Since $\cof (\I) \geq \non (\I)$ for any non-trivial ideal $\I$, $\cof (t^0) \geq \cc$ follows, and the main
problem about cofinalities of tree ideals is whether they can be equal to $\cc$ or must be strictly above $\cc$.

The question whether $\cof (\ell^0)$ and $\cof (m^0)$ are larger than $\cc$ was discussed in private
communication with M. De\v co and M. Repick\'y, and Repick\'y~\cite{Reta} in the meantime used our
result to obtain a characterization of $\cof (\ell^0)$ as $\dd ((\ell^0)^\cc)$.

\bigskip

\noindent {\bf Acknowledgments.} This research was started in February 2015 while the first author was visiting the other two
authors at the KGRC of the University of Vienna and at the Vienna University of Technology. He thanks 
the JSPS for its support and the universities in Vienna for their hospitality. Our work was continued in November 2015
at the University of East Anglia (first and second authors) and at the Isaac Newton Institute for Mathematical Sciences (INI)
(first and third authors). We thank the INI for its support and both institutions for their hospitality.



\section{The disjoint maximal antichain property}

\begin{defi}
Let $\TT$ be a combinatorial tree forcing. $\TT$ has the {\em disjoint maximal antichain property} if there is a maximal antichain
$(T_\alpha : \alpha < \cc )$ in $\TT$ such that $[T_\alpha] \cap [T_\beta] = \emptyset$ for all $\alpha \neq \beta$.
\end{defi}

The following has been known for some time (see also~\cite[Theorem 1.2]{Reta}).

\begin{prop}  \label{dmap}
Assume $\TT$ has the disjoint maximal antichain property. Then $cf (\cof (t^0)) > \cc$.
\end{prop}

\begin{proof}
Let $(T_\alpha : \alpha < \cc )$ be a disjoint maximal antichain in $\TT$.
Also let $\kappa = cf (\cof (t^0))$ and assume $\kappa \leq \cc$. We shall derive a contradiction.
Assume $\X_\alpha \sub t^0$, $\alpha < \kappa$, are of size  $< \cof (t^0)$. We shall show that
$\X = \bigcup \{ \X_\alpha : \alpha < \kappa \}$ is not cofinal in $t^0$. By homogeneity of the
tree forcing $\TT$, we know that $\X_\alpha$ is not cofinal below $T_\alpha$, that is,
there is $X_\alpha \in t^0$, $X_\alpha \sub [T_\alpha ]$, such that $X_\alpha \not\sub Y$
for all $Y \in \X_\alpha$. Let $X = \bigcup \{ X_\alpha : \alpha < \kappa \}$. By disjointness
of the maximal antichain, we see that $X \in t^0$. Obviously $X \not\sub Y$ for all
$Y \in \X$, and we are done.
\end{proof}

Note that for only showing $\cof (t^0) > \cc$, the homogeneity of the forcing is not needed
(that is, properties 1, 2, and 5 of Definition 1 are enough).

\begin{defi}
Let $\TT$ be a combinatorial tree forcing. $\TT$ has the {\em incompatibility shrinking property} if 
for any $T \in \TT$ and any family $(S_\alpha: \alpha < \mu)$, $\mu < \cc$, in $\TT$ such that $S_\alpha$ is incompatible with $T$ for all $\alpha$, 
one can find $T' \leq T$ such that $[T']$ is disjoint from all the $[S_\alpha]$.
\end{defi}

\begin{prop}   \label{isp}
Let $\TT$ be a combinatorial tree forcing. The incompatibility shrinking property for $\TT$ implies the
disjoint maximal antichain property for $\TT$. In fact, it implies that any maximal antichain can be refined to
a disjoint maximal antichain.
\end{prop}

\begin{proof}
Let $(T_\alpha : \alpha < \cc)$ be a dense set of conditions in $\TT$ all of which lie below a given maximal antichain of size $\cc$.
We construct $A \sub \cc$ of size $\cc$ and $\{ S_\alpha : \alpha \in A \} \sub \TT$ such that
\begin{itemize}
\item $S_\alpha \leq T_\alpha$ for $\alpha \in A$,
\item if $\alpha \notin A$, then $T_\alpha$ is compatible with some $S_\beta$ for $\beta < \alpha$
   with $\beta \in A$,
\item $[S_\alpha ] \cap [S_\beta] = \emptyset$ for $\alpha \neq \beta$ from $A$.
\end{itemize}
Clearly, these conditions imply that $(S_\alpha : \alpha \in A)$ is a disjoint maximal antichain. Also $A$ must necessarily have size $\cc$.

Suppose we are at stage $\alpha < \cc$ of the construction. If $T_\alpha$ is compatible with some $S_\beta$
where $\beta < \alpha$, $\beta \in A$, let $\alpha \notin A$, and we are done. If this is not the
case, let $\alpha \in A$. By the incompatibility shrinking property we find $T' = S_\alpha$ as required.
\end{proof}

\begin{ex}
Sacks forcing $\SS$,  Mathias forcing $\RR$, and Silver forcing $\VV$ have the incompatibility shrinking property
and thus also the disjoint maximal antichain property.
\end{ex}

To see this, simply use that for any two incompatible $S,T \in \SS$, $[S] \cap [T]$ is at most countable, while in the
case of $\VV$, this intersection is finite and for $\RR$, even empty.

From this we obtain that $cf (\cof (s^0)) > \cc$~\cite[Theorem 1.3]{JMS92} where
$s^0$ is the Marczewski ideal, that $cf (\cof (r^0)) > \cc$~\cite{Ma} where $r^0$ is
the ideal of nowhere Ramsey sets, and that $cf (\cof (v^0)) > \cc$
where $v^0$ is the Silver ideal.

We also note that if there is a fusion argument for $\TT$, then the continuum hypothesis CH implies the
incompatibility shrinking property and thus also the disjoint maximal antichain property. For Laver and Miller
forcings, a weaker hypothesis is sufficient.

\begin{prop}   \label{Laver-dmap}
Assume $\bb = \cc$. Then Laver forcing $\LL$ has the incompatibility shrinking property
and thus also the disjoint maximal antichain property.
\end{prop}

\begin{proof}
Fix $T \in \LL$, $\mu < \cc$, and any family $(S_\alpha: \alpha < \mu)$ in $\LL$ such that $S_\alpha$ is incompatible with $T$ for all $\alpha$.
Since $T \cap S_\alpha$ does not contain a Laver tree, by \cite[Lemma 2.3]{GRSS95}, there is a function $g_\alpha : \omlom \to \omega$ such that
if $x \in [T] \cap [S_\alpha]$, then there are infinitely many $n$ with $x(n) < g_\alpha (x \re n)$.
By $\bb = \cc$, there is $f : \omlom \to \omega$ eventually dominating all $g_\alpha$. 
Let $T' = \{ s\in T : s(n) > f (s \re n)$ for all $n \in \dom (s)$ beyond the stem of $T \}$. 
Clearly $T'$ is still a Laver tree with the same stem as $T$. Furthermore, $[T' ] \cap
[S_\alpha] = \emptyset$ for if $x$ belonged to the intersection, we would have 
$x(n) < g_\alpha (x\re n)$ for infinitely many $n$ and $x (n) > f( x\re n)$ for all $n$ beyond the
stem of $T'$, a contradiction. 
\end{proof}

A similar argument which we leave to the reader shows:

\begin{prop}   \label{Miller-dmap}
Assume $\dd = \cc$. Then Miller forcing $\MM$ has the incompatibility shrinking property
and thus also the disjoint maximal antichain property.
\end{prop}

\begin{ques}
Do $\LL$ or $\MM$ have the disjoint maximal antichain property in ZFC?
\end{ques}



\section{The selective disjoint antichain property}

We now consider a property weaker than the disjoint maximal antichain property which is sufficient
to show that the cofinalities of the Laver ideal $\ell^0$ and the Miller ideal $m^0$ are larger than
$\cc$ in ZFC.

\begin{defi}
Let $\TT$ be a combinatorial tree forcing. $\TT$ has the {\em selective disjoint antichain property} if there is an antichain
$(T_\alpha : \alpha < \cc )$ in $\TT$ such that 
\begin{itemize}
\item $[T_\alpha] \cap [T_\beta] = \emptyset$ for all $\alpha \neq \beta$,
\item for all $T \in \TT$ there is $S \leq T$ such that
\begin{itemize}
\item either $S \leq T_\alpha$ for some $\alpha < \cc$,
\item or $| [S] \cap [T_\alpha ] | \leq 1$ for all $\alpha < \cc$.
\end{itemize}
\end{itemize}
\end{defi}

We note that for our applications, it would be enough to have $| [S] \cap [T_\alpha ] | \leq \aleph_0$ in the last clause.

\begin{thm} \label{sdap}
Assume $\TT$ has the selective disjoint antichain property. Then $cf (\cof (t^0)) > \cc$.
\end{thm}

\begin{proof}
Let $(T_\alpha : \alpha < \cc )$ be a selective disjoint antichain in $\TT$. Also assume that $(S_\beta : \beta < \cc)$
is a list of all trees $S$ in $\TT$ such that $| [S] \cap [T_\alpha] | \leq 1$ for all $\alpha < \cc$. 
Put $\kappa = cf (\cof (t^0))$ and assume $\kappa \leq \cc$.
Also assume $\X_\alpha \sub t^0$, $\alpha < \kappa$, are of size  $< \cof (t^0)$. As in the proof of
Proposition~\ref{dmap}, we shall show that
$\X = \bigcup \{ \X_\alpha : \alpha < \kappa \}$ is not cofinal in $t^0$. 

By ``large disjoint antichains", we find $T_\alpha ' \leq T_\alpha$ such that $[T_\alpha ' ] \cap
\bigcup_{\beta < \alpha } [ S_\beta] = \emptyset$. By homogeneity, there is $X_\alpha \in t^0$
with $X_\alpha \sub [T_\alpha ']$ such that $X_\alpha \not\sub Y$ for all $Y \in \X_\alpha$.
Let $X = \bigcup \{ X_\alpha : \alpha < \kappa \}$. Obviously $X \not\sub Y$ for all
$Y \in \X$. We need to show that $X$ belongs to $t^0$. To this end, let $T \in \TT$.

First assume there is $S \leq T$ such that $S \leq T_\alpha$ for some $\alpha < \cc$.
Then $[S] \cap X \sub X_\alpha$. Since $X_\alpha \in t^0$, there is $S' \leq S$
such that $[S'] \cap X_\alpha = \emptyset$, and $[S'] \cap X = \emptyset$ follows.

Next assume there is $S \leq T$ such that $| [S] \cap [T_\alpha]| \leq 1$ for all $\alpha < \cc$.
Then $S = S_\beta$ for some $\beta < \cc$. By construction, we know that $X_\alpha \cap
[S_\beta] = \emptyset$ for $\alpha > \beta$. Hence $[S_\beta] \cap X \sub \bigcup_{\alpha \leq \beta}
[S_\beta] \cap [T_\alpha]$ and therefore $| [S_\beta] \cap X | < \cc$. Using again
``large disjoint antichains", we see that there is $S' \leq S_\beta$ such that $[S'] \cap X
= \emptyset$, as required. This completes the proof of the theorem.
\end{proof}

Again note that for only showing $\cof (t^0) > \cc$, the homogeneity of the forcing is not needed
(that is, properties 1, 2, and 5 of Definition 1 are enough)..

The next property of a combinatorial tree forcing $\TT$ implies that $\TT$ adds a minimal
real and, in fact, standard proofs of minimality go via this property.

\begin{defi}
Let $\TT$ be a combinatorial tree forcing. $\TT$ has the {\em constant or one-to-one property} if for all $T \in \TT$ and all continuous
$f : [T] \to \twoom$, there is $S \leq T$ such that $f\re [S]$ is either constant or one-to-one.
\end{defi}

It is known that both Miller forcing and Laver forcing have the constant or one-to-one property.
For the former, this is implicit in work of Miller~\cite[Section 2]{Mi84}, for the latter, in work of
Gray~\cite{Gra80} (see also~\cite[Theorems 2 and 7]{Gro87} for similar arguments). These results are formulated in
terms of minimality. For completeness' sake, we include a proof of the more difficult case of
Laver forcing in our formulation. Note also that the result for $\MM$ is a trivial consequence of the result for $\LL$.

\begin{thm}[Miller]  \label{Miller-min}
Miller forcing $\MM$ has the constant or one-to-one property.
\end{thm}

\begin{thm}[Gray]  \label{Laver-min}
Laver forcing $\LL$ has the constant or one-to-one property.
\end{thm}

\begin{proof}
Fix $f$ and $T$. The pure decision property of Laver forcing implies:

\begin{sclaim}  \label{claim1}
Let $n \in \omega$ and $\tau \in T$ with $\stem (T) \sub \tau$. There are $T' \leq_0 T_\tau$ and $s \in 2^n$ such that
$[T'] \sub f^{-1} ( [s] )$.
\end{sclaim}

\begin{sclaim}  \label{claim2}
Let $\tau \in T$ with $\stem (T) \sub \tau$. There are $T' \leq_0 T_\tau$ and $x = x_\tau \in \twoom$ such that
if $(k^n_\tau : n \in \omega)$ is the increasing enumeration of $\succ_{T'} (\tau)$ then
$[T'_{\tau \ha k^n_\tau}] \sub f^{-1} ( [x \re ( |\tau|]  + n))$.
\end{sclaim}

\begin{proof}
Using Claim~\ref{claim1}, construct a $\leq_0$-decreasing sequence $(S^n : n \in \omega )$ with $S^0 \leq_0 T_\tau$  
and a $\subset$-increasing sequence $(s^n \in 2^{n + |\tau|} : n \in \omega )$ such that $[S^n ] \sub f^{-1} ( [s^n])$
for all $n$. Let $k^n_\tau = \min ( \succ_{S^n} (\tau) \sem (k^{n-1}_\tau + 1) )$ where we put $k^{-1}_\tau = -1$.
Let $T'$ be such that $\succ_{T'} (\tau) = \{ k^n_\tau : n \in \omega \}$ and $T'_{\tau \ha k^n_\tau} = S^n_{\tau \ha k^n_\tau}$.
Also let $x = \bigcup_n s^n \in \twoom$. Then $T' \leq_0 T_\tau$ and $[T'_{\tau \ha k^n_\tau}] = [S^n_{\tau \ha k^n_\tau}]
\sub [S^n] \sub f^{-1} ([s^n]) = f^{-1} ([ x \re ( |\tau|]  + n) )$.
\end{proof}

By Claim~\ref{claim2} and a fusion argument we see

\begin{sclaim}  \label{claim3}
There are $T' \leq_0 T$, $(x_\tau : \tau \in T' , \stem (T) \sub \tau )$, and $(( k^n_\tau : n \in \omega ) : \tau \in T' , \stem (T) \sub \tau )$
such that $(k^n_\tau : n \in\omega)$ is the increasing enumeration of $\succ_{T'} (\tau)$ for all $\tau$ and 
$[T'_{\tau \ha k^n_\tau}] \sub f^{-1} ( [x_\tau \re ( |\tau|]  + n) )$ for all $n$ and all $\tau$. In particular 
$[T'_\tau ] \sub f^{-1} ( [x_\tau \re |\tau| ] )$ for all $\tau$.
\end{sclaim}

The properties of the $x_\tau$ imply in particular that $x_{\tau \ha k^n_\tau}$ converges to $x_\tau$ as $n$
goes to infinity. Now define a rank function for $\tau \in T'$ as follows.
\begin{itemize}
\item $\rho (\tau) = 0 \;\Loleriar\; \exists^\infty k \in \succ_{T'} (\tau)$ such that $x_{\tau\ha k} \neq x _\tau$,
\item for $\alpha > 0$, $\rho (\tau) = \alpha \; \Loleriar \; \neg \rho (\tau) < \alpha \; \land \; \exists^\infty k \in \succ_{T'} (\tau) \;
   (\rho (\tau \ha k ) < \alpha )$.
\end{itemize}
By the convergence property of the $x_\tau$, we see that $\rho (\tau) = 0$ implies in particular that
the set $\{ x_{\tau \ha k^n_\tau} : n \in \omega \}$ is infinite. 

{\sf Case 1.} $\rho (\tau) = \infty$ for some $\tau \in T'$ (i.e., the rank is undefined). \\
Then we can easily construct a Laver tree $S \leq T'$ such that $\stem (S) = \tau$ and $x_\sigma = x_\tau$ for all $\sigma \in S$
with $\sigma \supseteq \tau$. We claim that $f \re [S]$ is constant with value $x_\tau$. Indeed let $y \in [S]$. Fix $k \geq |\tau|$.
By construction $y \in [S_{y \re k}] \sub f^{-1} ( [x_\tau \re k ])$. Since this holds for all $k$, $f(y) = x_\tau$, and we are done.

{\sf Case 2.} $\rho (\tau)$ is defined for all $\tau \in T'$. \\
Recall that $F \sub T'$ is a {\em front} if for all $y \in [T']$ there is a unique $n$ with $y \re n \in F$. We build a subtree
$S$ of $T'$ by specifying fronts $F_n$, $n \in \omega$, such that for every $\sigma \in F_{n + 1}$ there is a
(necessarily unique) $\tau \in F_n$ with $\tau \subset \sigma$. That is, $S$ will be the tree generated by the fronts:
$\sigma \in S$ iff there are $n \in \omega$ and $\tau \in F_n$ with $\sigma \sub \tau$. Additionally, we shall
guarantee that there are $s_\tau \sub x_\tau$ for $\tau \in \bigcup_n F_n$ such that
\begin{itemize}
\item if $\sigma \neq \sigma'$ both are in $F_n$ then $[s_\sigma ] \cap [s_{\sigma'}] = \emptyset$,
\item $[S_\tau] \sub f^{-1} ( [s_\tau])$ for $\tau \in \bigcup_n F_n$,
\item if $\sigma \subset \tau$ with $\sigma \in F_n$ and $\tau \in F_{n+1}$ then $s_\sigma \subset s_\tau$,
\item if $\sigma \subset \tau$ with $\sigma \in F_n$ and $\tau \in F_{n+1}$ then for every $k$ with $|\sigma| \leq k < | \tau |$,
   $x_{\tau \re k} = x_\sigma$ and $\rho (\sigma) > \rho (\tau \re |\sigma| + 1) > ... > \rho (\tau \re |\tau | - 1 ) = 0$.
\end{itemize}
We first verify that this is enough to guarantee that $f \re [S]$ is one-to-one. If $y,y' \in [S]$ are distinct, then
there are $n, i, i' \in \omega$ such that $y \re i$ and $y' \re i'$ are distinct elements of $F_n$. Then
$y \in f^{-1} ( [s_{y \re i} ] )$, $y' \in f^{-1} ( [s_{y' \re i'} ] )$ by the second clause, and $[s_{y \re i} ]$ and $[s_{y' \re i'} ]$
are disjoint by the first clause. Hence $f(y) \neq f(y')$ as required. Thus it suffices to construct the $F_n$ and $s_\tau$.

$n = 0.$ We let $F_0 = \{ \stem (T') \} = \{ \stem (S) \}$. Also let $s_{\stem (S) } = x_{\stem (S)} \re | \stem (S)|$.

Suppose $F_n$ and $s_\sigma$ for $\sigma \in F_n$ have been constructed. We shall construct $F_{n+1}$,
$s_\sigma$ for $\sigma \in F_{n+1}$, as well as the part of the tree $S$ in between $F_n$ and $F_{n+1}$.
Fix $\sigma \in F_n$. By $A_\sigma^n$ we denote the part of $S$ between $\sigma$ and $F_{n+1}$, that is,
$A_\sigma^n = \{ \tau \in S : \sigma \sub \tau$ and $\tau \subset \upsilon$ for some $\upsilon$ in $F_{n+1} \}$.
$A_\sigma^n$ will be constructed recursively so as to satisfy the forth clause above.

Put $\sigma$ into $A_\sigma^n$. Suppose some $\tau \supseteq \sigma$ has been put into $A_\sigma^n$,
$x_\tau = x_\sigma$ and, in case $\tau \supset \sigma$, $\rho (\sigma) > \rho (\tau)$. In case $\rho (\tau) =0$,
no successor of $\tau$ will be in $A_\sigma^n$ and the successors of $\tau$ will belong to $F_{n+1}$, as explained below.
If $\rho (\tau) > 0$, then $x_{\tau \ha k} = x_\tau$ for almost all $k \in \succ_{T'} (\tau)$ and $\rho (\tau \ha k) < \rho (\tau)$
for infinitely many $k \in \succ_{T'} (\tau)$. Hence we can prune the successor level of $\tau$ to $\succ_S (\tau)$ such that
$x_{\tau \ha k} = x_\tau$ and $\rho (\tau \ha k) < \rho (\tau)$ for all $k \in \succ_S (\tau)$. The forth clause is clearly
satisfied. This completes the construction of $A_\sigma^n$.

Now fix $\tau \in A_\sigma^n$ with $\rho (\tau) = 0$. By pruning $\succ_{T'} (\tau)$, if necessary, we may assume
without loss of generality that the $x_{\tau \ha k_\tau^m}$, $m\in\omega$, are all pairwise distinct and converge to $x_\tau = x_\sigma$ 
and that, in fact,
there is a strictly increasing sequence $(i_\tau^m : m \in \omega)$ such that $i_\tau^m = \min \{ i : x_{\tau \ha k_\tau^m} (i) \neq
x_\tau (i) \}$. Unfixing $\tau$, we may additionally assume that if $\tau \neq \tau'$ are both in $A_\sigma^n$ of rank $0$
and $m , m' \in \omega$, then $i_\tau^m \neq i_{\tau'}^{m'}$. Finally we may assume that all such $i_\tau^m$ are larger than
$| s_\sigma |$. This means in particular that $s_\sigma \sub x_{\tau \ha k_\tau^m}$ for all $\tau$ and $m$ because
$s_\sigma \sub x_\sigma = x_\tau$. Now choose $s_{\tau \ha k_\tau^m} \sub x_{\tau \ha k_\tau^m}$ such that 
$| s_{\tau \ha k_\tau^m} | > i_\tau^m$. Then $s_\sigma \subset s_{\tau \ha k_\tau^m}$ and the $s_{\tau \ha k_\tau^m}$
for distinct pairs $(\tau,m)$ with $\tau \in A_\sigma^n$ of rank $0$ and $m \in \omega$ are pairwise incompatible.

Unfix $\sigma \in F_n$. Let $F_{n+1} = \{ \tau \ha k_\tau^m : \tau \in A_\sigma^n$ for some $\sigma \in F_n, 
\rho (\tau) = 0$, and $m \in \omega \}$. The third clause is immediate. To see the first clause, take distinct
$\tau , \tau' \in F_{n+1}$. There are $\sigma , \sigma ' \in F_n$ such that $\sigma \subset \tau$ and $\sigma ' \subset \tau '$.
If $\sigma \neq \sigma '$, then $[s_\tau] \cap [s_{\tau'}] = \emptyset$ because $[s_\sigma] \cap [s_{\sigma'}] = \emptyset$
and $s_\sigma \subset s_\tau$ and $s_{\sigma '} \subset s_{\tau '}$. If $\sigma = \sigma '$, then
$[s_\tau] \cap [s_{\tau'}] = \emptyset$ by the construction in the previous paragraph. Finally, to see the second clause, by pruning $T'_\tau$
for $\tau \in F_{n+1}$ if necessary, we may assume $[T'_\tau] \sub f^{-1} ( [s_\tau])$ (see Claim~\ref{claim3}). This completes the recursive
construction and the proof of the theorem.
\end{proof}

\begin{prop}   \label{constant-11}
Assume $\TT$ is a combinatorial tree forcing with the constant or one-to-one property. Then $\TT$ has the selective disjoint 
antichain property.
\end{prop}

\begin{proof}
Let $f : \omom \to \twoom $ be a continuous function witnessing ``large disjoint antichains" of $\TT$. 
Let $\{ x_\alpha : \alpha < \cc \}$ be an enumeration of $\twoom$. Let $T_\alpha \in \TT$ be
such that $[T_\alpha] = f^{-1} ( \{ x_\alpha \} )$. We check that $(T_\alpha : \alpha < \cc)$ witnesses
the selective disjoint antichain property. Clearly $[T_\alpha] \cap [T_\beta] = \emptyset$ for $\alpha
\neq \beta$. Given $T \in \TT$, find $S \leq T$ such that $f \re [S]$ is constant or one-to-one.
In the first case, $S \leq T_\alpha$ for some $\alpha$, and in the second case, $| [S] \cap [T_\alpha] |
\leq 1$ for all $\alpha$, and we are done.
\end{proof}

We are finally ready to complete the proof of the main result of this note.

\begin{cor}   \label{Laver-Miller-cof}
$cf (\cof (\ell^0)) > \cc$ and $cf (\cof (m^0)) > \cc$.
\end{cor}

\begin{proof}
This follows from Theorem~\ref{Laver-min}, Theorem~\ref{Miller-min}, Proposition~\ref{constant-11},
and Theorem~\ref{sdap}.
\end{proof}



\section{Problems}

For some natural tree forcings, we still do not know whether the cofinality of the corresponding tree ideal is
larger than $\cc$ in ZFC. Let $\FM$ denote {\em full splitting Miller forcing}, originally introduced by~\cite{NR93}
(see also~\cite{KLta}), that is, conditions are Miller trees $T \sub \omlom$
such that whenever $s \in T$ is a splitting node, then $s \ha n \in T$ for all $n \in \omega$. 
$fm^0$ is the {\em full splitting Miller ideal}.

\begin{ques}  \label{ques2}
Is $\cof (fm^0) > \cc$?
\end{ques}

By the discussion in Section 2 (before Proposition~\ref{Laver-dmap}), we know this is true under CH.

More generally, one may ask:

\begin{ques}   \label{ques3}
Are there combinatorial tree forcings $\TT$ which consistently fail to have the disjoint maximal antichain property? Which consistently fail to 
satisfy $\cof (t^0) > \cc$? For which $t^0$ consistently has a Borel basis?
\end{ques}

Note that the existence of a Borel basis implies $\cof (t^0) = \cc$.
By the above comment $fm^0$ has no Borel basis under CH, but this is open in ZFC.
Question~\ref{ques3} is also of interest for tree forcings which do not necessarily satisfy all the clauses of
Definition 1, e.g., for non-homogeneous forcing notions.

By~\cite[Theorems 1.4 and 1.5]{JMS92}, we know that $\cof (s^0)$ can consistently assume arbitrary values $\leq 2^\cc$
whose cofinality is larger than $\cc$ and it is easy to see that the same arguments work for other tree ideals
like $m^0$ and $\ell^0$. (In these models CH holds.)

\begin{ques}  \label{ques4}
Can we consistently separate the cofinalities of different tree ideals? E.g., are $\cof (s^0) < \cof (m^0)$ or
$\cof (m^0) < \cof (s^0)$ consistent?
\end{ques}



\end{document}